\documentclass[11pt,a4paper]{article}

\usepackage{inputenc}
\usepackage{amsmath}
\usepackage{bm}
\usepackage{bbold}
\usepackage{amsthm}
\usepackage{enumerate}

\usepackage{hyperref}

\title{Extremal Properties of Tropical Eigenvalues and Solutions to Tropical Optimization Problems}

\author{Nikolai Krivulin\thanks{Faculty of Mathematics and Mechanics, Saint Petersburg State University, 28 Universitetsky Ave., Saint Petersburg, 198504, Russia, 
nkk@math.spbu.ru.}\thanks{The work was supported in part by the Russian Foundation for Humanities under Grant \#13-02-00338.}
}

\date{}

\newtheorem{theorem}{Theorem}
\newtheorem{lemma}[theorem]{Lemma}
\newtheorem{corollary}[theorem]{Corollary}

\begin{document}

\maketitle

\begin{abstract}
An unconstrained optimization problem is formulated in terms of tropical mathematics to minimize a functional that is defined on a vector set by a matrix and calculated through multiplicative conjugate transposition. For some particular cases, the minimum in the problem is known to be equal to the tropical spectral radius of the matrix. We examine the problem in the common setting of a general idempotent semifield. A complete direct solution in a compact vector form is obtained to this problem under fairly general conditions. The result is extended to solve new tropical optimization problems with more general objective functions and inequality constraints. Applications to real-world problems that arise in project scheduling are presented. To illustrate the results obtained, numerical examples are also provided.
\\

\\

\end{abstract}

\section{Introduction}

Tropical (idempotent) mathematics, which focuses on the theory and applications of idempotent semirings, dates back to a few works published in the early 1960s, such as seminal papers by Pandit \cite{Pandit1961Anew}, Cuninghame-Green \cite{Cuninghamegreen1962Describing}, Giffler \cite{Giffler1963Scheduling}, Vorob{'}ev \cite{Vorobjev1963Theextremal} and Romanovski{\u\i} \cite{Romanovskii1964Asymptotic}. Since these works, there have been many new important results obtained in this field and reported in various monographs, including the most recent books by Golan \cite{Golan2003Semirings}, Heidergott et al \cite{Heidergott2006Maxplus}, Gondran and Minoux \cite{Gondran2008Graphs}, Butkovi\v{c} \cite{Butkovic2010Maxlinear}, and in a great number of contributed papers.

Optimization problems that are formulated and solved in the tropical mathematics setting constitute an important research domain within the field. An optimization problem, which is drawn from machine scheduling, arose in the early paper by Cuninghame-Green \cite{Cuninghamegreen1962Describing} to minimize a functional defined on a vector set by a given matrix. In terms of the semifield $\mathbb{R}_{\max,+}$ with the usual maximum in the role of addition and arithmetic addition in the role of multiplication, this problem is represented in the form
\begin{equation*}
\begin{aligned}
&
\text{minimize}
&&
\bm{x}^{-}\bm{A}\bm{x},
\end{aligned}
\end{equation*}
where $\bm{A}$ is a square matrix, $\bm{x}$ is the unknown vector, $\bm{x}^{-}$ is the multiplicative conjugate transpose of $\bm{x}$, and the matrix-vector operations follow the usual rules with the scalar addition and multiplication defined by the semifield.

The solution to the problem is based on a useful property of tropical eigenvalues. It has been shown in \cite{Cuninghamegreen1962Describing} that the minimum in the problem coincides with the tropical spectral radius (maximum eigenvalue) of $\bm{A}$ and is attained at any tropical eigenvector corresponding to this radius. The same results in a different setting were obtained by Engel and Schneider \cite{Engel1975Diagonal}.

In recent years, the above optimization problem has appeared in various applied contexts, which motivates further development of related solutions. Specifically, the problem was examined by Elsner and van den Driessche \cite{Elsner2004Maxalgebra,Elsner2010Maxalgebra} in the framework of decision making. Applications to discrete event systems and location analysis were considered in Krivulin \cite{Krivulin2005Evaluation,Krivulin2009Methods,Krivulin2011Anextremal,Krivulin2012Anew}.

In tropical algebra, the spectral radius is given directly by a closed-form expression, which makes the evaluation of the minimum a rather routine task. By contrast, it may be not trivial to represent all solutions to the problem in terms of tropical mathematics in a direct, explicit form that is suitable for further analysis and applications.

Cuninghame-Green \cite{Cuninghamegreen1979Minimax} proposed a solution that is based on reducing to a linear programming problem. This solution, however, does not offer a direct representation of all solutions in terms of tropical mathematics. Furthermore, Elsner and van den Driessche \cite{Elsner2004Maxalgebra,Elsner2010Maxalgebra} indicated that, in addition to the eigenvectors, there are other vectors, which can also solve the problem. An implicit description of the complete solution set has been suggested in the form of a tropical linear inequality. Though a solution approach has been proposed, this approach provided particular solutions by means of a numerical algorithm, rather than a complete solution in an explicit form. Note finally that the existing approaches mainly concentrate on the solution in terms of the semifield $\mathbb{R}_{\max,+}$, and assume that the matrix $\bm{A}$ is irreducible. 

In this paper, we examine the problem in the common setting of a general idempotent semifield. A complete direct solution in a compact vector form is obtained to this problem under fairly general assumptions. We apply and further develop the algebraic technique introduced in \cite{Krivulin2004Onsolution,Krivulin2005Evaluation,Krivulin2006Solution,Krivulin2009Onsolution,Krivulin2009Methods}, which combines general methods of solution of tropical linear inequalities with results of the tropical spectral theory. On the basis of the solution to the above problem, we offer complete direct solutions to new problems with more complicated objective functions and linear inequality constraints. Applications of these results to real-world problems that arise in project scheduling and numerical examples are also provided.

The paper is organized as follows. We start with Section~\ref{S-pdar}, where a short, concise overview of notation, definitions, and preliminary results in tropical algebra is given to provide a common framework in terms of a general idempotent semifield for subsequent results. Complete solutions to various linear inequalities together with new direct proofs, which are of independent interest, are presented in Section~\ref{S-stli}. Furthermore, Section~\ref{S-epoeaop} offers a new complete direct solution to the above optimization problem and provides examples of known solutions to more general optimization problems. Finally, in Section~\ref{S-ntop}, two new tropical optimization problems are examined. We obtain complete solutions to the problems and illustrate these results with applications  and numerical examples.

\section{Preliminary definitions and results}
\label{S-pdar}

The purpose of this section is to offer an overview of the necessary notation, definitions, and preliminary results of tropical (idempotent) algebra in a short and concise manner. Both thorough exposition of and compendious introduction to the theory and methods of tropical mathematics are provided by many authors, including recent publications by Golan \cite{Golan2003Semirings}, Heidergott et al \cite{Heidergott2006Maxplus}, Akian et al \cite{Akian2007Maxplus}, Litvinov \cite{Litvinov2007Themaslov}, Gondran and Minoux \cite{Gondran2008Graphs}, Butkovi\v{c} \cite{Butkovic2010Maxlinear}. 

In the overview below, we consider a general idempotent semifield and mainly follow the notation and presentation in \cite{Krivulin2006Solution,Krivulin2009Methods,Krivulin2011Anextremal,Krivulin2012Anew,Krivulin2013Amultidimensional} to provide the basis for the derivation of subsequent results in the most possible common general form. Further details of the theory and methods, as well as related references can be found in the publications listed before.

\subsection{Idempotent semifield}

Suppose that $\mathbb{X}$ is a nonempty set with binary operations, addition $\oplus$ and multiplication $\otimes$. Addition and multiplication are associative and commutative, and have respective neutral elements, zero $\mathbb{0}$ and identity $\mathbb{1}$. Multiplication is distributive over addition and has $\mathbb{0}$ as absorbing element.

Furthermore, addition is idempotent, which means that $x\oplus x=x$ for all $x\in\mathbb{X}$. Multiplication is invertible in the sense that for each $x\ne\mathbb{0}$, there exists $x^{-1}$ such that $x^{-1}\otimes x=\mathbb{1}$. With these properties, the system $\langle\mathbb{X},\mathbb{0},\mathbb{1},\oplus,\otimes\rangle$ is generally identified as the idempotent semifield.

The power notation is routinely defined for all $x\ne\mathbb{0}$ and integer $p\geq1$ as $x^{p}=x^{p-1}\otimes x$, $x^{-p}=(x^{-1})^{p}$, $x^{0}=\mathbb{1}$, and $\mathbb{0}^{p}=\mathbb{0}$. Moreover, the integer power is held to extend to rational exponents, and thus the semifield is taken algebraically complete (radicable).

To save writing, the multiplication sign $\otimes$ is omitted from here on. The power notation is used in the sense of the above definition.

The idempotent addition defines a partial order such that $x\leq y$ if and only if $x\oplus y=y$ for $x,y\in\mathbb{X}$. With respect to this order, the addition possesses an extremal property, which implies that the inequalities $x\leq x\oplus y$ and $y\leq x\oplus y$ hold for all $x,y\in\mathbb{X}$. Both addition and multiplication are monotone in each argument, which means that the inequalities $x\leq y$ and $u\leq v$ imply $x\oplus u\leq y\oplus v$ and $xu\leq yv$, respectively.

The partial order is considered extendable to a total order to make the semifield linearly ordered. In what follows, the relation signs and optimization problems are thought of in terms of this linear order. 

Common examples of the idempotent semifield under consideration are $\mathbb{R}_{\max,+}=\langle\mathbb{R}\cup\{-\infty\},-\infty,0,\max,+\rangle$, $\mathbb{R}_{\min,+}=\langle\mathbb{R}\cup\{+\infty\},+\infty,0,\min,+\rangle$, $\mathbb{R}_{\max,\times}=\langle\mathbb{R}_{+}\cup\{0\},0,1,\max,\times\rangle$, and $\mathbb{R}_{\min,\times}=\langle\mathbb{R}_{+}\cup\{+\infty\},+\infty,1,\min,\times\rangle$, where $\mathbb{R}$ is the set of real numbers and $\mathbb{R}_{+}=\{x\in\mathbb{R}|x>0\}$.

Specifically, the semifield $\mathbb{R}_{\max,+}$ is endowed with an addition and multiplication, defined by the usual maximum operation and arithmetic addition, respectively. The number $-\infty$ is taken as the zero element, and $0$ as the identity. For each $x\in\mathbb{R}$, there exists an inverse $x^{-1}$, which coincides with $-x$ in ordinary arithmetic. The power $x^{y}$ is defined for all $x,y\in\mathbb{R}$ and equal to the arithmetic product $xy$. The order induced on $\mathbb{R}_{\max,+}$ by the idempotent addition corresponds to the natural linear order on $\mathbb{R}$.

In the semifield $\mathbb{R}_{\min,\times}$, the operations are defined as $\oplus=\min$ and $\otimes=\times$, and their neutral elements as $\mathbb{0}=+\infty$ and $\mathbb{1}=1$. Both inversion and exponentiation have the usual meaning. The relation $\leq$, defined by idempotent addition, corresponds to an order that is opposite to the natural order on $\mathbb{R}$.

\subsection{Matrix algebra}

The idempotent semifield is routinely generalized to idempotent systems of matrices and vectors. Let $\mathbb{X}^{m\times n}$ be the set of matrices over $\mathbb{X}$, with $m$ rows and $n$ columns. The matrix with all zero entries is the zero matrix denoted $\bm{0}$. Any matrix without zero rows (columns) is called row- (column-) regular.

Addition and multiplication of conforming matrices, and scalar multiplication follow the usual rules, where the operations $\oplus$ and $\otimes$ play the roles of the ordinary addition and multiplication, respectively. The extremal property of addition and the monotonicity of addition and multiplication in the carrier semifield are extended component-wise to the matrix operations.

For any matrix $\bm{A}$, its transpose is represented by $\bm{A}^{T}$.

Consider square matrices in the set $\mathbb{X}^{n\times n}$. A matrix that has $\mathbb{1}$ along the diagonal and $\mathbb{0}$ elsewhere is the identity matrix, which is denoted by $\bm{I}$. For any matrix $\bm{A}\in\mathbb{X}^{n\times n}$ and integer $p\geq1$, the matrix power is immediately defined by $\bm{A}^{p}=\bm{A}^{p-1}\bm{A}$ and $\bm{A}^{0}=\bm{I}$.

A matrix is reducible if identical permutation of its rows and columns can reduce it to a block-triangular normal form, and irreducible otherwise. The lower block-triangular normal form of a matrix $\bm{A}\in\mathbb{X}^{n\times n}$ is expressed as
\begin{equation}\label{E-MNF}
\bm{A}
=
\left(
	\begin{array}{cccc}
		\bm{A}_{11}	& \bm{0}			& \ldots	& \bm{0} \\
		\bm{A}_{21}	& \bm{A}_{22}	&					& \bm{0} \\
		\vdots			& \vdots			& \ddots	& \\
		\bm{A}_{s1}	& \bm{A}_{s2}	& \ldots	&	\bm{A}_{ss}
	\end{array}
\right),
\end{equation}
where $\bm{A}_{ii}$ is either irreducible or a zero matrix of order $n_{i}$, $\bm{A}_{ij}$ is an arbitrary matrix of size $n_{i}\times n_{j}$ for all $i=1,\ldots,s$, $j<i$, and $n_{1}+\cdots+n_{s}=n$.

The trace of any matrix $\bm{A}=(a_{ij})\in\mathbb{X}^{n\times n}$ is given by
$$
\mathop\mathrm{tr}\bm{A}
=
a_{11}\oplus\cdots\oplus a_{nn}.
$$

It directly results from the definition that, for any matrices $\bm{A}$ and $\bm{B}$, and scalar $x$, the following equalities hold: 
$$
\mathop\mathrm{tr}(\bm{A}\oplus\bm{B})
=
\mathop\mathrm{tr}\bm{A}
\oplus
\mathop\mathrm{tr}\bm{B},
\qquad
\mathop\mathrm{tr}(\bm{A}\bm{B})
=
\mathop\mathrm{tr}(\bm{B}\bm{A}),
\qquad
\mathop\mathrm{tr}(x\bm{A})
=
x\mathop\mathrm{tr}\bm{A}.
$$

The first equality is extended in \cite{Krivulin2012Atropical} to a binomial identity for traces, which is valid, for all $m\geq0$, in the form 
\begin{equation}
\mathop\mathrm{tr}(\bm{A}\oplus\bm{B})^{m}
=
\bigoplus_{k=1}^{m}\mathop{\bigoplus\hspace{2.3em}}_{i_{1}+\cdots+i_{k}=m-k}\mathop\mathrm{tr}(\bm{A}\bm{B}^{i_{1}}\cdots\bm{A}\bm{B}^{i_{k}})
\oplus
\mathop\mathrm{tr}\bm{B}^{m}.
\label{E-trABm}
\end{equation}

Any matrix, which has only one row (column) is a row (column) vector. The set of column vectors of order $n$ is denoted by $\mathbb{X}^{n}$. The vector with all zero elements is the zero vector denoted $\bm{0}$.

A vector is called regular if it has no zero elements. Suppose that $\bm{x}\in\mathbb{X}^{n}$ is a regular vector. Clearly, if $\bm{A}\in\mathbb{X}^{n\times n}$ is a row-regular matrix, then the vector $\bm{A}\bm{x}$ is regular, and if $\bm{A}$ is column-regular, then $\bm{x}^{T}\bm{A}$ is regular.

For any nonzero column vector $\bm{x}=(x_{i})\in\mathbb{X}^{n}$, its multiplicative conjugate transpose is a row vector $\bm{x}^{-}=(x_{i}^{-})$ that has elements $x_{i}^{-}=x_{i}^{-1}$ if $x_{i}\ne\mathbb{0}$, and $x_{i}^{-}=\mathbb{0}$ otherwise.

The following properties of the conjugate transposition are easy to verify. Let $\bm{x},\bm{y}\in\mathbb{X}^{n}$ be regular vectors. Then, the element-wise inequality $\bm{x}\leq\bm{y}$ results in $\bm{x}^{-}\geq\bm{y}^{-}$ and vice versa. Furthermore, the matrix inequality $\bm{x}\bm{y}^{-}\geq(\bm{x}^{-}\bm{y})^{-1}\bm{I}$ holds entry-wise. If $\bm{x}$ is a nonzero vector, then $\bm{x}^{-}\bm{x}=\mathbb{1}$.

A scalar $\lambda\in\mathbb{X}$ is an eigenvalue of a matrix $\bm{A}\in\mathbb{X}^{n\times n}$, if there exists a nonzero vector $\bm{x}\in\mathbb{X}^{n}$ such that
$$
\bm{A}\bm{x}
=
\lambda\bm{x}.
$$

Any vector $\bm{x}$ that provides this equality is an eigenvector of $\bm{A}$, which corresponds to $\lambda$.

Every irreducible matrix has only one eigenvalue. The eigenvectors of irreducible matrices have no zero components and so are regular. 

Reducible matrices may possess several eigenvalues. The maximal eigenvalue (in the sense of the order on $\mathbb{X}$) of a matrix $\bm{A}\in\mathbb{X}^{n\times n}$ is called the spectral radius of the matrix, and is directly calculated as
$$
\lambda
=
\bigoplus_{m=1}^{n}\mathop\mathrm{tr}\nolimits^{1/m}(\bm{A}^{m}).
$$

From this equality, it follows, in particular, that the inequality $\lambda^{m}\geq\mathop\mathrm{tr}\bm{A}^{m}$ is valid for all $m=1,\ldots,n$.

\section{Solution to linear inequalities}
\label{S-stli}

This section presents complete solutions to linear inequalities to be used in the solving of optimization problems. These results can be obtained as consequences of the solutions of related equations in \cite{Krivulin2004Onsolution,Krivulin2006Solution,Krivulin2009Onsolution,Krivulin2009Methods,Krivulin2013Amultidimensional}. Below, we offer the solutions provided with new direct independent proofs, which are of separate interest.

Suppose that, given a matrix $\bm{A}\in\mathbb{X}^{m\times n}$ and a regular vector $\bm{d}\in\mathbb{X}^{m}$, the problem is to find all vectors $\bm{x}\in\mathbb{X}^{n}$ that satisfy the inequality
\begin{equation}
\bm{A}\bm{x}
\leq
\bm{d}.
\label{I-Axd}
\end{equation}

A solution to the problem is provided by the next statement. 
\begin{lemma}
\label{L-IxdA}
For any column-regular matrix $\bm{A}$ and regular vector $\bm{d}$, all solutions to \eqref{I-Axd} are given by
\begin{equation}
\bm{x}
\leq
(\bm{d}^{-}\bm{A})^{-}.
\label{I-xdA}
\end{equation}
\end{lemma}
\begin{proof}
It suffices to verify that inequalities \eqref{I-Axd} and \eqref{I-xdA} are consequences of each other. First, we take inequality \eqref{I-Axd} and multiply it by $(\bm{d}^{-}\bm{A})^{-}\bm{d}^{-}$ from the left. By applying properties of the conjugate transposition, we write
$$
\bm{x}
\leq
(\bm{d}^{-}\bm{A})^{-}\bm{d}^{-}\bm{A}\bm{x}
\leq
(\bm{d}^{-}\bm{A})^{-}\bm{d}^{-}\bm{d}
=
(\bm{d}^{-}\bm{A})^{-},
$$
and thus obtain \eqref{I-xdA}. At the same time, the left multiplication of inequality \eqref{I-xdA} by the matrix $\bm{A}$ and the above mentioned properties give
$$
\bm{A}\bm{x}
\leq
\bm{A}(\bm{d}^{-}\bm{A})^{-}
\leq
\bm{d}\bm{d}^{-}\bm{A}(\bm{d}^{-}\bm{A})^{-}
=
\bm{d},
$$
which yields \eqref{I-Axd} and completes the proof.
\end{proof}

We now consider the following problem. Given a matrix $\bm{A}\in\mathbb{X}^{n\times n}$, find all vectors $\bm{x}\in\mathbb{X}^{n}$ such that
\begin{equation}
\bm{A}\bm{x}
\leq
\bm{x}
\label{I-Axx}.
\end{equation}

To describe a solution to inequality \eqref{I-Axx} in a compact form, we introduce a function that maps each matrix $\bm{A}\in\mathbb{X}^{n\times n}$ onto the scalar
$$
\mathop\mathrm{Tr}(\bm{A})
=
\mathop\mathrm{tr}\bm{A}\oplus\cdots\oplus\mathop\mathrm{tr}\bm{A}^{n},
$$
and use the star operator (the Kleene star), which takes $\bm{A}$ to the matrix
$$
\bm{A}^{\ast}
=
\bm{I}\oplus\bm{A}\oplus\cdots\oplus\bm{A}^{n-1}.
$$

The derivation of the solution is based on the result, which was apparently first suggested by Carr{\'e} \cite{Carre1971Analgebra} and will be referred to as the Carr{\'e} inequality. In terms of the above defined function, the result states that any matrix $\bm{A}$ with $\mathop\mathrm{Tr}(\bm{A})\leq\mathbb{1}$ satisfies the inequality
$$
\bm{A}^{k}
\leq
\bm{A}^{\ast},
\qquad
k
\geq
0.
$$

A complete direct solution to inequality \eqref{I-Axx} is given as follows.

\begin{lemma}\label{L-xAu}
For any matrix $\bm{A}$ with $\mathop\mathrm{Tr}(\bm{A})\leq\mathbb{1}$, all solutions to inequality \eqref{I-Axx} are given by
$$
\bm{x}
=
\bm{A}^{\ast}\bm{u},
\qquad
\bm{u}\in\mathbb{X}^{n}.
$$
\end{lemma}
\begin{proof}
First, we show that, for any $\bm{u}\in\mathbb{X}^{n}$, the vector $\bm{x}=\bm{A}^{\ast}\bm{u}$ satisfies \eqref{I-Axx}. It follows from the Carr{\'e} inequality that $\bm{A}\bm{A}^{\ast}=\bm{A}\oplus\cdots\oplus\bm{A}^{n}\leq\bm{A}^{\ast}$. Then,
$$
\bm{A}\bm{x}
=
\bm{A}(\bm{A}^{\ast}\bm{u})
=
(\bm{A}\bm{A}^{\ast})\bm{u}
\leq
\bm{A}^{\ast}\bm{u}
=
\bm{x}.
$$

Now suppose that $\bm{x}$ is a solution to inequality \eqref{I-Axx}, and then verify that $\bm{x}=\bm{A}^{\ast}\bm{u}$ for some $\bm{u}\in\mathbb{X}^{n}$. Indeed, left multiplication of \eqref{I-Axx} by $\bm{A}$ yields the inequality $\bm{A}^{m}\bm{x}\leq\bm{x}$ for all integer $m\geq1$, and thus the inequality $\bm{A}^{\ast}\bm{x}\leq\bm{x}$. Because $\bm{A}^{\ast}\geq\bm{I}$, we also have the inequality $\bm{A}^{\ast}\bm{x}\geq\bm{x}$. Both inequalities result in the equality $\bm{x}=\bm{A}^{\ast}\bm{u}$, where $\bm{u}=\bm{x}$.
\end{proof}

Finally, we suppose that, given a matrix $\bm{A}\in\mathbb{X}^{n\times n}$ and a vector $\bm{b}\in\mathbb{X}^{n}$, we need to obtain all regular vectors $\bm{x}\in\mathbb{X}^{n}$ that satisfy the inequality
\begin{equation}
\bm{A}\bm{x}\oplus\bm{b}
\leq
\bm{x}.
\label{I-Axbx}
\end{equation}

The next complete direct solution to the problem is given in \cite{Krivulin2013Amultidimensional}.
\begin{theorem}\label{T-Axbx}
For any matrix $\bm{A}$ and vector $\bm{b}$, the following statements hold:
\begin{enumerate}
\item If $\mathop\mathrm{Tr}(\bm{A})\leq\mathbb{1}$, then all regular solutions to inequality \eqref{I-Axbx} are given by $\bm{x}=\bm{A}^{\ast}\bm{u}$, where $\bm{u}$ is any regular vector such that $\bm{u}\geq\bm{b}$.
\item If $\mathop\mathrm{Tr}(\bm{A})>\mathbb{1}$, then there is no regular solution.
\end{enumerate}
\end{theorem}

\section{Extremal properties of eigenvalues and optimization problems}
\label{S-epoeaop}

The spectral radius (maximum eigenvalue) of a matrix in idempotent algebra exhibits an extremal property in the sense that it represents the minimum of a functional defined by the matrix and calculated with a conjugate transposition operator.

To be more specific, let $\bm{A}\in\mathbb{X}^{n\times n}$ be a matrix with spectral radius $\lambda$ and $\bm{x}\in\mathbb{X}^{n}$ be a vector. We form the functional $\bm{x}^{-}\bm{A}\bm{x}$ and consider the problem
\begin{equation}
\begin{aligned}
&
\text{minimize}
&&
\bm{x}^{-}\bm{A}\bm{x},
\end{aligned}
\label{P-minxAx}
\end{equation}
where the minimum is taken over all regular vectors $\bm{x}$.

Then, it can be shown that the minimum in this problem is equal to $\lambda$.

Although this property has been known in various settings for a long time, the problem of finding all vectors $\bm{x}$ that yield the minimum did not admit a complete solution in terms of a general semifield. 

In the context of tropical mathematics, this property has been first investigated in \cite{Cuninghamegreen1962Describing} in the framework of the semifield $\mathbb{R}_{\max,+}$. It has been shown that the minimum in \eqref{P-minxAx} is attained at any eigenvector of the matrix $\bm{A}$, which corresponds to $\lambda$. Furthermore, the problem was reduced in \cite{Cuninghamegreen1979Minimax} to a linear programming problem, which, however, did not offer a direct representation of the solution vectors in terms of tropical mathematics. Similar results in a different and somewhat more general setting have been obtained in \cite{Engel1975Diagonal}.

More solution vectors for the problem were indicated in \cite{Elsner2004Maxalgebra,Elsner2010Maxalgebra}. Specifically, it has been shown that not only the eigenvectors, which solve the equation $\bm{A}\bm{x}=\lambda\bm{x}$, but all the vectors, which satisfy the inequality $\bm{A}\bm{x}\leq\lambda\bm{x}$, yield the minimum of the functional $\bm{x}^{-}\bm{A}\bm{x}$. However, the solution, which has been given to the last inequality, had the form of a numerical algorithm rather than a direct explicit form.

A complete direct solution to problem \eqref{P-minxAx} is obtained in \cite{Krivulin2012Atropical,Krivulin2013Amultidimensional} as a consequence of the solution to more general optimization problems. Below, we provide the solution to the problem under general assumptions and present a new independent proof as a good example of the approach.
\begin{lemma}\label{L-minxAx}
Let $\bm{A}$ be a matrix with spectral radius $\lambda>\mathbb{0}$. Then, the minimum in \eqref{P-minxAx} is equal to $\lambda$ and all regular solutions of the problem are given by
$$
\bm{x}
=
(\lambda^{-1}\bm{A})^{\ast}\bm{u},
\qquad
\bm{u}\in\mathbb{X}^{n}.
$$
\end{lemma}
\begin{proof}
To verify the statement, we show that $\lambda$ is a lower bound for the objective function in \eqref{P-minxAx}, and then obtain all regular vectors that produce this bound. First, assume that the matrix $\bm{A}$ is irreducible and thus has only regular eigenvectors. We take an arbitrary eigenvector $\bm{x}_{0}$ and write
$$
\bm{x}^{-}\bm{A}\bm{x}
=
\bm{x}^{-}\bm{A}\bm{x}\bm{x}_{0}^{-}\bm{x}_{0}
\geq
\bm{x}^{-}\bm{A}\bm{x}_{0}(\bm{x}^{-}\bm{x}_{0})^{-1}
=
\lambda\bm{x}^{-}\bm{x}_{0}(\bm{x}^{-}\bm{x}_{0})^{-1}
=
\lambda.
$$

Suppose the matrix $\bm{A}$ is reducible and has the block-triangular form \eqref{E-MNF}. Let $\lambda_{i}$ be the eigenvalue of the matrix $\bm{A}_{ii}$, $i=1,\ldots,s$, and $\lambda=\lambda_{1}\oplus\cdots\oplus\lambda_{s}$. Furthermore, we represent the vector $\bm{x}$ in the block form $\bm{x}^{T}=(\bm{x}_{1}^{T},\ldots,\bm{x}_{s}^{T})$, where $\bm{x}_{i}$ is a vector of order $n_{i}$. It follows from the above result for irreducible matrices that
$$
\bm{x}^{-}\bm{A}\bm{x}
=
\bigoplus_{i=1}^{s}\bigoplus_{j=1}^{s}\bm{x}_{i}^{-}\bm{A}_{ij}\bm{x}_{j}
\geq
\bigoplus_{i=1}^{s}\bm{x}_{i}^{-}\bm{A}_{ii}\bm{x}_{i}
\geq
\bigoplus_{i=1}^{s}\lambda_{i}
=
\lambda.
$$

Consider the equation $\bm{x}^{-}\bm{A}\bm{x}=\lambda$. Since $\lambda$ is a lower bound, this equation can be replaced by the inequality $\bm{x}^{-}\bm{A}\bm{x}\leq\lambda$. Furthermore, the set of regular solutions of the inequality, if they exist, coincides with that of the inequality $\lambda^{-1}\bm{A}\bm{x}\leq\bm{x}$. Indeed, after the left multiplication of the former inequality by $\lambda^{-1}\bm{x}$, we have $\lambda^{-1}\bm{A}\bm{x}\leq\lambda^{-1}\bm{x}\bm{x}^{-}\bm{A}\bm{x}\leq\bm{x}$, which results in the latter. At the same time, left multiplication of the latter inequality by $\lambda\bm{x}^{-}$ directly leads to the former one.

With the condition $\mathop\mathrm{Tr}(\lambda^{-1}\bm{A})=\lambda^{-1}\mathop\mathrm{tr}\bm{A}\oplus\cdots\oplus\lambda^{-n}\mathop\mathrm{tr}\bm{A}^{n}\leq\mathbb{1}$, we apply Lemma~\ref{L-xAu} to the last inequality, which yields the desired complete solution to problem \eqref{P-minxAx}.
\end{proof}

To conclude this section, we provide examples of solutions to more general optimization problems examined in \cite{Krivulin2012Acomplete,Krivulin2014Aconstrained}. We start with an unconstrained problem, which has an extended objective function and formulated as follows. Given a matrix $\bm{A}\in\mathbb{X}^{n\times n}$ and vectors $\bm{p},\bm{q}\in\mathbb{X}^{n}$, find all regular vectors $\bm{x}\in\mathbb{X}^{n}$ such that
\begin{equation}
\begin{aligned}
&
\text{minimize}
&&
\bm{x}^{-}\bm{A}\bm{x}\oplus\bm{x}^{-}\bm{p}\oplus\bm{q}^{-}\bm{x}.
\end{aligned}
\label{P-xAxxpqx}
\end{equation}

A direct solution to the problem was proposed in \cite{Krivulin2012Acomplete}. For an irreducible matrix $\bm{A}$ with spectral radius $\lambda>\mathbb{0}$ and a regular vector $\bm{q}$, it was suggested that the minimum value in \eqref{P-xAxxpqx} is equal to
$$
\mu
=
\lambda\oplus(\bm{q}^{-}\bm{p})^{1/2},
$$
and all regular solutions of the problem are given by
$$
\bm{x}
=
(\mu^{-1}\bm{A})^{\ast}\bm{u},
\qquad
\mu^{-1}\bm{p}
\leq
\bm{u}
\leq
\mu(\bm{q}^{-}(\mu^{-1}\bm{A})^{\ast})^{-}.
$$

Below, we offer a solution to a more general problem, which includes problem \eqref{P-xAxxpqx} as a special case. The results obtained to problem \eqref{P-xAxxpqx} in the next section show that the above solution appears to be incomplete.

Finally, consider the following constrained version of problem \eqref{P-minxAx}. Suppose that, given matrices $\bm{A},\bm{B}\in\mathbb{X}^{n\times n}$, $\bm{C}\in\mathbb{X}^{m\times n}$ and vectors $\bm{g}\in\mathbb{X}^{n}$, $\bm{h}\in\mathbb{X}^{m}$, we need to find all regular vectors $\bm{x}\in\mathbb{X}^{n}$ such that
\begin{equation}
\begin{aligned}
&
\text{minimize}
&&
\bm{x}^{-}\bm{A}\bm{x},
\\
&
\text{subject to}
&&
\bm{B}\bm{x}\oplus\bm{g}
\leq
\bm{x},
\\
&
&&
\bm{C}\bm{x}
\leq
\bm{h}.
\end{aligned}
\label{P-xAxBxgCxh}
\end{equation}

Let $\bm{A}$ be a matrix with spectral radius $\lambda>\mathbb{0}$, $\bm{B}$ be a matrix with $\mathop\mathrm{Tr}(\bm{B})\leq\mathbb{1}$, $\bm{C}$ be a column-regular matrix, and $\bm{h}$ be a regular vector such that $\bm{h}^{-}\bm{C}\bm{B}^{\ast}\bm{g}\leq\mathbb{1}$. Then, problem \eqref{P-xAxBxgCxh} can be solved as in \cite{Krivulin2014Aconstrained} to find that the minimum value in problem \eqref{P-xAxBxgCxh} is equal to
$$
\theta
=
\bigoplus_{k=1}^{n}\mathop{\bigoplus\hspace{0.0em}}_{0\leq i_{0}+i_{1}+\cdots+i_{k}\leq n-k}\mathop\mathrm{tr}\nolimits^{1/k}(\bm{B}^{i_{0}}(\bm{A}\bm{B}^{i_{1}}\cdots\bm{A}\bm{B}^{i_{k}})(\bm{I}\oplus\bm{g}\bm{h}^{-}\bm{C})),
$$
and all regular solutions of the problem are given by
$$
\bm{x}
=
(\theta^{-1}\bm{A}\oplus\bm{B})^{\ast}\bm{u},
\qquad
\bm{g}
\leq
\bm{u}
\leq
(\bm{h}^{-}\bm{C}(\theta^{-1}\bm{A}\oplus\bm{B})^{\ast})^{-}.
$$

\section{New tropical optimization problems}
\label{S-ntop}

In this section we offer complete direct solutions for two new problems with objective functions that involve the functional $\bm{x}^{-}\bm{A}\bm{x}$. The proposed proofs of the results demonstrate the main ideas and techniques of the approach introduced in \cite{Krivulin2004Onsolution,Krivulin2005Evaluation,Krivulin2009Onsolution,Krivulin2009Methods} and further developed in \cite{Krivulin2011Anextremal,Krivulin2012Acomplete,Krivulin2012Anew,Krivulin2012Atropical,Krivulin2012Solution,Krivulin2013Amultidimensional,Krivulin2014Aconstrained}, which is based on the solutions to linear inequalities and the extremal property of the spectral radius discussed above.

\subsection{Unconstrained problem}

We start with an unconstrained problem, which is a further extension of the problems considered in the short conference papers \cite{Krivulin2012Solution,Krivulin2012Acomplete}, where less general objective functions were examined. The solution presented below expands the technique proposed in these papers to solve the new problem, but suggests a more compact representation of the results, which is based on the solution to the linear inequalities with arbitrary matrices in \cite{Krivulin2013Amultidimensional}. As a consequence, a more accurate solution to the problem in \cite{Krivulin2012Acomplete} is obtained.

Given a matrix $\bm{A}\in\mathbb{X}^{n\times n}$, vectors $\bm{p},\bm{q}\in\mathbb{X}^{n}$, and a scalar $c\in\mathbb{X}$, the problem is to find all regular vectors $\bm{x}\in\mathbb{X}^{n}$ such that
\begin{equation}
\begin{aligned}
&
\text{minimize}
&&
\bm{x}^{-}\bm{A}\bm{x}\oplus\bm{x}^{-}\bm{p}\oplus\bm{q}^{-}\bm{x}\oplus c.
\end{aligned}
\label{P-xAxxpqxc}
\end{equation}

\subsubsection{Solution to unconstrained problem}

A complete direct solution to the problem under fairly general conditions is given in a compact vector form by the following result.

\begin{theorem}\label{T-xAxxpqxc}
Let $\bm{A}$ be a matrix with spectral radius $\lambda>\mathbb{0}$, and $\bm{q}$ be a regular vector. Denote
\begin{equation}
\mu
=
\lambda
\oplus
\bigoplus_{m=1}^{n}
(\bm{q}^{-}\bm{A}^{m-1}\bm{p})^{1/(m+1)}
\oplus
c.
\label{E-mulambdaqApc}
\end{equation}

Then, the minimum in \eqref{P-xAxxpqxc} is equal to $\mu$ and all regular solutions of the problem are given by
\begin{equation*}
\bm{x}
=
(\mu^{-1}\bm{A})^{\ast}\bm{u},
\qquad
\mu^{-1}\bm{p}
\leq
\bm{u}
\leq
\mu(\bm{q}^{-}(\mu^{-1}\bm{A})^{\ast})^{-}.
\end{equation*}
\end{theorem}
\begin{proof}
Below, we replace problem \eqref{P-xAxxpqxc} with an equivalent constrained problem that is to minimize a scalar $\mu$ subject to the constraint
\begin{equation}
\bm{x}^{-}\bm{A}\bm{x}\oplus\bm{x}^{-}\bm{p}\oplus\bm{q}^{-}\bm{x}\oplus c
\leq
\mu,
\label{I-xAxxpqxc}
\end{equation}
where the minimum is taken over all regular vectors $\bm{x}$.

Our use of inequality \eqref{I-xAxxpqxc} is twofold: first, we consider $\mu$ as a parameter and solve the inequality with respect to $\bm{x}$, and second, we exploit \eqref{I-xAxxpqxc} to derive a sharp lower bound on $\mu$. 

First, we note that inequality \eqref{I-xAxxpqxc} is itself equivalent to the system of inequalities
$$
\bm{x}^{-}\bm{A}\bm{x}
\leq
\mu,
\qquad
\bm{x}^{-}\bm{p}
\leq
\mu,
\qquad
\bm{q}^{-}\bm{x}
\leq
\mu,
\qquad
c
\leq
\mu.
$$

It follows from the first inequality and Lemma~\ref{L-minxAx} that $\mu\geq\bm{x}^{-}\bm{A}\bm{x}\geq\lambda>\mathbb{0}$.

Considering that the vectors $\bm{x}$ and $\bm{q}$ are regular, we apply Lemma~\ref{L-IxdA} to the first three inequalities to bring the system into the form
$$
\bm{A}\bm{x}
\leq
\mu\bm{x},
\qquad
\bm{p}
\leq
\mu\bm{x},
\qquad
\bm{x}
\leq
\mu\bm{q},
\qquad
c
\leq
\mu.
$$

The second and third inequalities together give $\bm{p}\leq\mu^{2}\bm{q}$. After multiplication by $\bm{q}^{-}$ from the left, we obtain another bound $\mu\geq(\bm{q}^{-}\bm{p})^{1/2}$.

Together with the forth inequality, we have
\begin{equation}
\mu
\geq
\lambda\oplus(\bm{q}^{-}\bm{p})^{1/2}\oplus c.
\label{I-muqpc}
\end{equation}

Furthermore, we obtain all regular vectors $\bm{x}$ that satisfy the first three inequalities of the system. We multiply the first two inequalities by $\mu^{-1}$ and then combine them into one to write the system of two inequalities
\begin{equation}
\mu^{-1}\bm{A}\bm{x}
\oplus
\mu^{-1}\bm{p}
\leq
\bm{x},
\qquad
\bm{x}
\leq
\mu\bm{q}.
\label{I-muAxmupx_xmuq}
\end{equation}

Note that $\mathop\mathrm{Tr}(\mu^{-1}\bm{A})\leq\mathop\mathrm{Tr}(\lambda^{-1}\bm{A})\leq\mathbb{1}$ since $\mu\geq\lambda$. Then, Theorem~\ref{T-Axbx} provides a complete solution to the first inequality at \eqref{I-muAxmupx_xmuq} in the form
$$
\bm{x}
=
(\mu^{-1}\bm{A})^{\ast}\bm{u},
\qquad
\bm{u}
\geq
\mu^{-1}\bm{p}.
$$

Substitution of the solution into the second inequality results in the system of inequalities defined in terms of the new variable $\bm{u}$ as
$$
(\mu^{-1}\bm{A})^{\ast}\bm{u}
\leq
\mu\bm{q},
\qquad
\bm{u}
\geq
\mu^{-1}\bm{p}.
$$

By solving the first inequality in the new system by means of Lemma~\ref{L-IxdA}, we arrive at the double inequality
$$
\mu^{-1}\bm{p}
\leq
\bm{u}
\leq
\mu(\bm{q}^{-}(\mu^{-1}\bm{A})^{\ast})^{-}.
$$

The set of vectors $\bm{u}$ defined by this inequality is not empty if and only if
$$
\mu^{-1}\bm{p}
\leq
\mu(\bm{q}^{-}(\mu^{-1}\bm{A})^{\ast})^{-}.
$$

By multiplying this inequality by $\mu^{-1}\bm{q}^{-}(\mu^{-1}\bm{A})^{\ast}$ from the left, we obtain
$$
\mu^{-2}\bm{q}^{-}(\mu^{-1}\bm{A})^{\ast}\bm{p}
\leq
\mathbb{1}.
$$

Since left multiplication of the latter inequality by $\mu(\bm{q}^{-}(\mu^{-1}\bm{A})^{\ast})^{-}$ gives $\mu^{-1}\bm{p}\leq\mu^{-1}(\bm{q}^{-}(\mu^{-1}\bm{A})^{\ast})^{-}\bm{q}^{-}(\mu^{-1}\bm{A})^{\ast}\bm{p}\leq\mu(\bm{q}^{-}(\mu^{-1}\bm{A})^{\ast})^{-}$, and so yields the former inequality, both inequalities are equivalent.

Consider the left-hand side of the last inequality and write it in the form
$$
\mu^{-2}\bm{q}^{-}\left(\bigoplus_{m=0}^{n-1}\mu^{-m}\bm{A}^{m}\right)\bm{p}
=
\bigoplus_{m=0}^{n-1}\mu^{-m-2}\bm{q}^{-}\bm{A}^{m}\bm{p}.
$$

Then, the inequality is equivalent to the system of inequalities
$$
\mu^{-m-2}\bm{q}^{-}\bm{A}^{m}\bm{p}
\leq
\mathbb{1},
\qquad
m=0,\ldots,n-1,
$$
which can further be rewritten as
$$
\mu^{m+1}
\geq
\bm{q}^{-}\bm{A}^{m-1}\bm{p},
\qquad
m=1,\ldots,n.
$$

Furthermore, we solve every inequality in the system to obtain
$$
\mu
\geq
(\bm{q}^{-}\bm{A}^{m-1}\bm{p})^{1/(m+1)},
\qquad
m=1,\ldots,n.
$$

By combining these solutions with the bounds at \eqref{I-muqpc}, we have
$$
\mu
\geq
\lambda
\oplus
\bigoplus_{m=1}^{n}
(\bm{q}^{-}\bm{A}^{m-1}\bm{p})^{1/(m+1)}
\oplus
c.
$$

As we need to find the minimum value of $\mu$, the last inequality must hold as equality \eqref{E-mulambdaqApc}. The solution set is then given by
$$
\bm{x}
=
(\mu^{-1}\bm{A})^{\ast}\bm{u},
\qquad
\mu^{-1}\bm{p}
\leq
\bm{u}
\leq
\mu(\bm{q}^{-}(\mu^{-1}\bm{A})^{\ast})^{-},
$$
which completes the proof.
\end{proof}

Note that a complete solution to problem \eqref{P-xAxxpqx} can now be obtained as a direct consequence of the above result when $c=\mathbb{0}$.
\begin{corollary}
Let $\bm{A}$ be a matrix with spectral radius $\lambda>\mathbb{0}$, and $\bm{q}$ be a regular vector. Denote
$$
\mu
=
\lambda
\oplus
\bigoplus_{m=1}^{n}
(\bm{q}^{-}\bm{A}^{m-1}\bm{p})^{1/(m+1)}.
$$

Then, the minimum in \eqref{P-xAxxpqx} is equal to $\mu$ and all regular solutions of the problem are given by
$$
\bm{x}
=
(\mu^{-1}\bm{A})^{\ast}\bm{u},
\qquad
\mu^{-1}\bm{p}
\leq
\bm{u}
\leq
\mu(\bm{q}^{-}(\mu^{-1}\bm{A})^{\ast})^{-}.
$$
\end{corollary}

\subsubsection{Application to project scheduling}

In this section, we provide an application example for the result obtained above to solve a problem, which is drawn from project scheduling \cite{Demeulemeester2002Project,Tkindt2006Multicriteria}. 

Suppose that a project involves a set of activities (jobs, tasks) to be performed in parallel according to precedence relationships given in the form of start-finish, late start and early finish temporal constraints. The start-finish constraints do not enable an activity to be completed until specified times have elapsed after the initiation of other activities. The activities are completed as soon as possible within these constraints.

The late start and the early finish constraints determine, respectively, the lower and upper boundaries of time windows that are allocated to the activities in the project. Each activity has to occupy its time window entirely. If the initiation time of the activity falls to the right of the lower bound of the window, the time is adjusted by shifting to this bound. In a similar way, the completion time is shifted to the upper bound of the window if this time appears to the left of the bound. 

For each activity in the project, the flow (turnaround, processing) time is defined as the time interval between the adjusted initiation and completion times of the activity. The optimal scheduling problem is to find the initiation times that minimize the maximum flow time over all activities, subject to the precedence constraints described above.

Consider a project with $n$ activities. For each activity $i=1,\ldots,n$, we denote the initiation time by $x_{i}$ and the completion time by $y_{i}$. Let $a_{ij}$ be the minimum possible time lag between the initiation of activity $j=1,\ldots,n$ and the completion of $i$. If the time lag is not specified for some $j$, we assume that $a_{ij}=-\infty$. The start-finish constraints lead to the equalities
$$
y_{i}
=
\max(a_{i1}+x_{1},\ldots,a_{in}+x_{n}),
\qquad
i=1,\ldots,n.
$$

Furthermore, we denote the late start and early finish times for activity $i$ by $q_{i}$ and $p_{i}$, respectively. Let $s_{i}$ be the adjusted initiation time and $t_{i}$ be the adjusted completion time of activity $i$. Considering the minimum time window defined by the late start and the early finish constraints, we have 
$$
s_{i}
=
\min(x_{i},q_{i})
=
-
\max(-x_{i},-q_{i}),
\quad
t_{i}
=
\max(y_{i},p_{i}),
\qquad
i=1,\ldots,n.
$$

Finally, with the maximum flow time, which is given by
$$
\max(t_{1}-s_{1},\ldots,t_{n}-s_{n}),
$$
we arrive at the optimal scheduling problem in the form
\begin{equation*}
\begin{aligned}
&
\text{minimize}
&&
\max_{1\leq i\leq n}(t_{i}-s_{i}),
\\
&
\text{subject to}
&&
s_{i}
=
-
\max(-x_{i},-q_{i}),
\\
&&&
t_{i}
=
\max\left(\max_{1\leq j\leq n}(a_{ij}+x_{j}),p_{i}\right),
\qquad
i=1,\ldots,n.
\end{aligned}
\end{equation*}

As the problem formulation involves only the operations of maximum, ordinary addition, and additive inversion, we can represent the problem in terms of the semifield $\mathbb{R}_{\max,+}$. First, we introduce the matrix-vector notation
$$
\bm{A}
=
(a_{ij}),
\quad
\bm{p}
=
(p_{i}),
\quad
\bm{q}
=
(q_{i}),
\quad
\bm{x}
=
(x_{i}),
\quad
\bm{s}
=
(s_{i}),
\quad
\bm{t}
=
(t_{i}).
$$

By using matrix algebra over $\mathbb{R}_{\max,+}$, we write
$$
\bm{s}
=
(\bm{x}^{-}\oplus\bm{q}^{-})^{-},
\qquad
\bm{t}
=
\bm{A}\bm{x}\oplus\bm{p}.
$$

Then, the objective function takes the form
$$
\bm{s}^{-}\bm{t}
=
(\bm{x}^{-}\oplus\bm{q}^{-})(\bm{A}\bm{x}\oplus\bm{p})
=
\bm{x}^{-}\bm{A}\bm{x}
\oplus
\bm{q}^{-}\bm{A}\bm{x}
\oplus
\bm{x}^{-}\bm{p}
\oplus
\bm{q}^{-}\bm{p}.
$$

The problem under study now reduces to the unconstrained problem
\begin{equation}
\begin{aligned}
&
\text{minimize}
&&
\bm{x}^{-}\bm{A}\bm{x}
\oplus
\bm{q}^{-}\bm{A}\bm{x}
\oplus
\bm{x}^{-}\bm{p}
\oplus
\bm{q}^{-}\bm{p},
\end{aligned}
\label{P-xAxqxxpqp}
\end{equation}
which has the form of \eqref{P-xAxxpqxc}.

The application of Theorem~\ref{T-xAxxpqxc} with $\bm{q}^{-}$ replaced by $\bm{q}^{-}\bm{A}$ and $c$ by $\bm{q}^{-}\bm{p}$ gives the following result.
\begin{theorem}
Let $\bm{A}$ be a matrix with spectral radius $\lambda>\mathbb{0}$, and $\bm{q}$ be a regular vector. Denote
\begin{equation}
\mu
=
\lambda
\oplus
\bigoplus_{m=0}^{n}
(\bm{q}^{-}\bm{A}^{m}\bm{p})^{1/(m+1)}.
\label{E-mulambdaqAp}
\end{equation}

Then, the minimum in \eqref{P-xAxqxxpqp} is equal to $\mu$ and all regular solutions of the problem are given by
\begin{equation*}
\bm{x}
=
(\mu^{-1}\bm{A})^{\ast}\bm{u},
\qquad
\mu^{-1}\bm{p}
\leq
\bm{u}
\leq
\mu(\bm{q}^{-}\bm{A}(\mu^{-1}\bm{A})^{\ast})^{-}.
\end{equation*}
\end{theorem}

\subsubsection{Numerical example}

Consider an example project that involves $n=3$ activities. Suppose that the matrix $\bm{A}$ and the vectors $\bm{p}$ and $\bm{q}$ are given by
$$
\bm{A}
=
\left(
\begin{array}{rrc}
2 & 4 & \mathbb{0}
\\
2 & 2 & 1
\\
0 & -1 & 1
\end{array}
\right),
\qquad
\bm{q}
=
\left(
\begin{array}{c}
1
\\
1
\\
1
\end{array}
\right),
\qquad
\bm{p}
=
\left(
\begin{array}{c}
3
\\
3
\\
3
\end{array}
\right),
$$
where the symbol $\mathbb{0}=-\infty$ is used to save writing.

To obtain the spectral radius $\lambda$, we successively calculate
$$
\bm{A}^{2}
=
\left(
\begin{array}{rrc}
6 & 6 & 5
\\
4 & 6 & 3
\\
2 & 4 & 2
\end{array}
\right),
\qquad
\bm{A}^{3}
=
\left(
\begin{array}{rrc}
8 & 10 & 7
\\
8 & 8 & 7
\\
6 & 6 & 5
\end{array}
\right),
\qquad
\lambda
=
3.
$$

Furthermore, we calculate
$$
\bm{q}^{-}\bm{A}
=
\left(
\begin{array}{ccc}
1 & 3 & 0
\end{array}
\right),
\qquad
\bm{q}^{-}\bm{A}^{2}
=
\left(
\begin{array}{ccc}
5 & 5 & 4
\end{array}
\right),
\qquad
\bm{q}^{-}\bm{A}^{3}
=
\left(
\begin{array}{ccc}
7 & 9 & 6
\end{array}
\right),
$$
and then obtain
$$
\bm{q}^{-}\bm{p}
=
2,
\qquad
\bm{q}^{-}\bm{A}\bm{p}
=
6,
\qquad
\bm{q}^{-}\bm{A}^{2}\bm{p}
=
8,
\qquad
\bm{q}^{-}\bm{A}^{3}\bm{p}
=
12,
\qquad
\mu
=
3.
$$

After evaluating the matrices
$$
\mu^{-1}\bm{A}
=
\left(
\begin{array}{rrr}
-1 & 1 & \mathbb{0}
\\
-1 & -1 & -2
\\
-3 & -4 & -1
\end{array}
\right),
\qquad
\mu^{-2}\bm{A}^{2}
=
\left(
\begin{array}{rrr}
0 & 0 & -1
\\
-2 & 0 & -3
\\
-4 & -2 & -2
\end{array}
\right),
$$
we can find
$$
(\mu^{-1}\bm{A})^{\ast}
=
\left(
\begin{array}{rrr}
0 & 1 & -1
\\
-1 & 0 & -2
\\
-3 & -2 & 0
\end{array}
\right),
\qquad
\bm{q}^{-}\bm{A}(\mu^{-1}\bm{A})^{\ast}
=
\left(
\begin{array}{ccc}
2 & 3 & 1
\end{array}
\right).
$$

By calculating the lower and upper bounds on the vector $\bm{u}$, we get
$$
\mu^{-1}\bm{p}
=
\left(
\begin{array}{c}
0
\\
0
\\
0
\end{array}
\right),
\qquad
\mu(\bm{q}^{-}\bm{A}(\mu^{-1}\bm{A})^{\ast})^{-}
=
\left(
\begin{array}{c}
1
\\
0
\\
2
\end{array}
\right).
$$

In terms of ordinary operations, the solution is given by the equalities
\begin{align*}
x_{1}
&=
\max(u_{1},u_{2}+1,u_{3}-1),
\\
x_{2}
&=
\max(u_{1}-1,u_{2},u_{3}-2),
\\
x_{3}
&=
\max(u_{1}-3,u_{2}-2,u_{3}),
\end{align*}
where the elements of the vector $\bm{u}$ satisfy the conditions
$$
0
\leq
u_{1}
\leq
1,
\qquad
u_{2}
=
0,
\qquad
0
\leq
u_{3}
\leq
2.
$$

Substitution of the bounds on $\bm{u}$ into the above equalities gives the result
$$
x_{1}
=
1,
\qquad
x_{2}
=
0,
\qquad
0\leq x_{3}\leq 2.
$$

\subsection{Constrained problem}

We now consider a new constrained problem which combines the objective function of the problem in \cite{Krivulin2012Solution} with inequality constraints from the problem in \cite{Krivulin2013Amultidimensional}. To solve the problem, we follow the method proposed in \cite{Krivulin2012Atropical}, which introduces an additional variable and then reduces the problem to the solving of a linear inequality with a matrix parametrized by the new variable.

Given matrices $\bm{A},\bm{B}\in\mathbb{X}^{n\times n}$ and vectors $\bm{p},\bm{g}\in\mathbb{X}^{n}$, we need to find all regular vectors $\bm{x}\in\mathbb{X}^{n}$ such that
\begin{equation}
\begin{aligned}
&
\text{minimize}
&&
\bm{x}^{-}\bm{A}\bm{x}\oplus\bm{x}^{-}\bm{p},
\\
&
\text{subject to}
&&
\bm{B}\bm{x}\oplus\bm{g}
\leq
\bm{x}.
\end{aligned}
\label{P-xAxxpBxgx}
\end{equation}

Note that, due to Theorem~\ref{T-Axbx}, the set of regular vectors that satisfy the inequality constraints in problem \eqref{P-xAxxpBxgx} is not empty if and only if $\mathop\mathrm{Tr}(\bm{B})\leq\mathbb{1}$.

\subsubsection{Solution to constrained problem}

The next assertion provides a direct complete solution to the problem.

\begin{theorem}\label{T-xAxxpBxgx}
Let $\bm{A}$ be a matrix with spectral radius $\lambda>\mathbb{0}$ and $\bm{B}$ be a matrix with $\mathop\mathrm{Tr}(\bm{B})\leq\mathbb{1}$. Denote
\begin{equation}
\theta
=
\lambda
\oplus
\bigoplus_{k=1}^{n-1}\mathop{\bigoplus\hspace{1.2em}}_{1\leq i_{1}+\cdots+i_{k}\leq n-k}\mathop\mathrm{tr}\nolimits^{1/k}(\bm{A}\bm{B}^{i_{1}}\cdots\bm{A}\bm{B}^{i_{k}}).
\label{E-theta}
\end{equation}

Then, the minimum in \eqref{P-xAxxpBxgx} is equal to $\theta$ and all regular solutions of the problem are given by
$$
\bm{x}
=
(\theta^{-1}\bm{A}\oplus\bm{B})^{\ast}\bm{u},
\qquad
\bm{u}
\geq
\theta^{-1}\bm{p}\oplus\bm{g}.
$$
\end{theorem}
\begin{proof}
First, we note that $\bm{x}^{-}\bm{A}\bm{x}\oplus\bm{x}^{-}\bm{p}\geq\bm{x}^{-}\bm{A}\bm{x}\geq\lambda>\mathbb{0}$, and thus the objective function in problem \eqref{P-xAxxpBxgx} is bounded from below. We now denote by $\theta$ the minimum of the objective function over all regular vectors $\bm{x}$ and verify that the value of $\theta$ is given by \eqref{E-theta}.

The set of all regular vectors $\bm{x}$ that yield the minimum in \eqref{P-xAxxpBxgx} is given by the system
$$
\bm{x}^{-}\bm{A}\bm{x}\oplus\bm{x}^{-}\bm{p}
=
\theta,
\qquad
\bm{B}\bm{x}\oplus\bm{g}
\leq
\bm{x}.
$$

Since $\theta$ is the minimum, the solution set does not change if we replace the first equation by the inequality
$$
\bm{x}^{-}\bm{A}\bm{x}\oplus\bm{x}^{-}\bm{p}
\leq
\theta.
$$

Furthermore, the last inequality is equivalent to the inequality
$$
\theta^{-1}\bm{A}\bm{x}\oplus\theta^{-1}\bm{p}
\leq
\bm{x}.
$$

Indeed, by multiplying the former inequality by $\theta^{-1}\bm{x}$ from the left, we obtain
$$
\theta^{-1}\bm{A}\bm{x}\oplus\theta^{-1}\bm{p}
\leq
\theta^{-1}\bm{x}\bm{x}^{-}\bm{A}\bm{x}\oplus\theta^{-1}\bm{x}\bm{x}^{-}\bm{p}
\leq
\bm{x},
$$
which yields the latter. On the other hand, the left multiplication of the latter inequality by $\theta\bm{x}^{-}$ immediately results in the former one.

The system, which determines all regular solutions, now becomes
$$
\theta^{-1}\bm{A}\bm{x}\oplus\theta^{-1}\bm{p}
\leq
\bm{x},
\qquad
\bm{B}\bm{x}\oplus\bm{g}
\leq
\bm{x}.
$$

The above two inequalities are equivalent to one inequality in the form
\begin{equation}
(\theta^{-1}\bm{A}\oplus\bm{B})\bm{x}
\oplus
\theta^{-1}\bm{p}\oplus\bm{g}
\leq
\bm{x}.
\label{I-thetaABxthetapgx}
\end{equation}

It follows from Theorem~\ref{T-Axbx} that the existence condition of regular solutions to inequality \eqref{I-thetaABxthetapgx} is given by the inequality
$$
\mathop\mathrm{Tr}(\theta^{-1}\bm{A}\oplus\bm{B})
\leq
\mathbb{1}.
$$

The application of binomial identity \eqref{E-trABm} to the left-hand side of the inequality, followed by the rearrangements of terms leads to the inequality
\begin{multline*}
\mathop\mathrm{Tr}(\theta^{-1}\bm{A}\oplus\bm{B})
=
\bigoplus_{m=1}^{n}\mathop\mathrm{tr}(\theta^{-1}\bm{A}\oplus\bm{B})^{m}
\\
=
\bigoplus_{m=1}^{n}\bigoplus_{k=1}^{m}\mathop{\bigoplus\hspace{2.3em}}_{i_{1}+\cdots+i_{k}=m-k}\theta^{-k}\mathop\mathrm{tr}(\bm{A}\bm{B}^{i_{1}}\cdots\bm{A}\bm{B}^{i_{k}})
\oplus
\mathop\mathrm{Tr}(\bm{B})
\\
=
\bigoplus_{k=1}^{n}\mathop{\bigoplus\hspace{1.2em}}_{0\leq i_{1}+\cdots+i_{k}\leq n-k}\theta^{-k}\mathop\mathrm{tr}(\bm{A}\bm{B}^{i_{1}}\cdots\bm{A}\bm{B}^{i_{k}})
\oplus
\mathop\mathrm{Tr}(\bm{B}).
\end{multline*}

Since $\mathop\mathrm{Tr}(\bm{B})\leq\mathbb{1}$ by the conditions of the theorem, the existence condition reduces to the inequalities
$$
\mathop{\bigoplus\hspace{1.2em}}_{0\leq i_{1}+\cdots+i_{k}\leq n-k}\theta^{-k}\mathop\mathrm{tr}(\bm{A}\bm{B}^{i_{1}}\cdots\bm{A}\bm{B}^{i_{k}})
\leq
\mathbb{1},
\qquad
k=1,\ldots,n.
$$

After solving the inequalities with respect to $\theta$, and combining of the obtained solutions into one inequality, we obtain
$$
\theta
\geq
\lambda
\oplus
\bigoplus_{k=1}^{n-1}\mathop{\bigoplus\hspace{1.2em}}_{1\leq i_{1}+\cdots+i_{k}\leq n-k}\mathop\mathrm{tr}\nolimits^{1/k}(\bm{A}\bm{B}^{i_{1}}\cdots\bm{A}\bm{B}^{i_{k}}).
$$

Taking into account that $\theta$ is the minimum of the objective function, we replace the last inequality by an equality, which gives \eqref{E-theta}.

By applying Theorem~\ref{T-Axbx} to write the solution for inequality \eqref{I-thetaABxthetapgx}, we finally obtain.
$$
\bm{x}
=
(\theta^{-1}\bm{A}\oplus\bm{B})^{\ast}\bm{u},
\qquad
\bm{u}
\geq
\theta^{-1}\bm{p}\oplus\bm{g}.
\qedhere
$$
\end{proof}

\subsubsection{Application to project scheduling}

Consider a project that consists of $n$ activities operating under start-finish, start-start, early start and early finish temporal constraints. The start-finish constraints determine the minimum allowed time lag between the initiation of one activity and the completion of another activity. The start-start constraints dictate the minimum lag between the initiation times of the activities. The early start constraints define the earliest possible times to initiate the activities. The early finish constraints require that the activities cannot be completed earlier than prescribed times. All activities in the project are completed as soon as possible to meet the constraints. The problem is to minimize the flow time over all activities.

In the similar manner as before, for each activity $i=1,\ldots,n$, we denote the initiation time by $x_{i}$ and the completion time by $y_{i}$. Given the start-finish time lags $a_{ij}$ and the earliest finish times $p_{i}$, we have
$$
y_{i}
=
\max(a_{i1}+x_{1},\ldots,a_{in}+x_{n},p_{i}),
\qquad
i=1,\ldots,n.
$$

Let $b_{ij}$ be the minimum allowed time lag between the initiation of activity $j=1,\ldots,n$ and the initiation of activity $i$, and $g_{i}$ be the earliest acceptable time of the initiation of $i$. The start-start and early start constraints give the inequalities
$$
x_{i}
\geq
\max(b_{i1}+x_{1},\ldots,a_{in}+x_{n},g_{i}),
\qquad
i=1,\ldots,n.
$$

After adding the objective function, we represent the scheduling problem in the form
\begin{equation*}
\begin{aligned}
&
\text{minimize}
&&
\max_{1\leq i\leq n}(y_{i}-x_{i}),
\\
&
\text{subject to}
&&
y_{i}
=
\max\left(\max_{1\leq j\leq n}(a_{ij}+x_{j}),p_{i}\right),
\\
&&&
x_{i}
\geq
\max\left(\max_{1\leq j\leq n}(b_{ij}+x_{j}),g_{i}\right),
\qquad
i=1,\ldots,n.
\end{aligned}
\end{equation*}

To rewrite the problem in terms of the semifield $\mathbb{R}_{\max,+}$, we define the following matrices and vectors
$$
\bm{A}
=
(a_{ij}),
\qquad
\bm{B}
=
(b_{ij}),
\qquad
\bm{p}
=
(p_{i}),
\qquad
\bm{g}
=
(g_{i}),
\qquad
\bm{x}
=
(x_{i}).
$$

In vector form, the problem becomes
\begin{equation*}
\begin{aligned}
&
\text{minimize}
&&
\bm{x}^{-}(\bm{A}\bm{x}\oplus\bm{p}),
\\
&
\text{subject to}
&&
\bm{B}\bm{x}\oplus\bm{g}
\leq
\bm{x}.
\end{aligned}
\end{equation*}

Clearly, the obtained problem is of the same form as problem \eqref{P-xAxxpBxgx} and thus admits a complete solution given by Theorem~\ref{T-xAxxpBxgx}.

\subsection{Numerical example}

To illustrate the result, we take a project with $n=3$ activities under constraints given by
$$
\bm{A}
=
\left(
\begin{array}{ccc}
4 & 0 & \mathbb{0}
\\
2 & 3 & 1
\\
1 & 1 & 3
\end{array}
\right),
\quad
\bm{B}
=
\left(
\begin{array}{rrr}
\mathbb{0} & -2 & 1
\\
0 & \mathbb{0} & 2
\\
-1 & \mathbb{0} & \mathbb{0}
\end{array}
\right),
\quad
\bm{p}
=
\left(
\begin{array}{c}
6
\\
6
\\
6
\end{array}
\right),
\quad
\bm{g}
=
\left(
\begin{array}{c}
1
\\
2
\\
3
\end{array}
\right).
$$

Then, we need to verify the existence condition for regular solutions in Theorem~\ref{T-xAxxpBxgx}. We have
$$
\bm{B}^{2}
=
\left(
\begin{array}{rrr}
0 & \mathbb{0} & 0
\\
1 & -2 & 1
\\
\mathbb{0} & -3 & 0
\end{array}
\right),
\qquad
\bm{B}^{3}
=
\left(
\begin{array}{rrr}
-1 & -2 & 1
\\
0 & -1 & 2
\\
-1 & \mathbb{0} & -1
\end{array}
\right),
\qquad
\mathop\mathrm{Tr}(\bm{B})
=
0.
$$

In addition, we obtain
$$
\bm{A}^{2}
=
\left(
\begin{array}{ccc}
8 & 4 & 1
\\
6 & 6 & 4
\\
5 & 4 & 6
\end{array}
\right),
\qquad
\bm{A}^{3}
=
\left(
\begin{array}{ccc}
12 & 8 & 5
\\
10 & 9 & 7
\\
9 & 7 & 9
\end{array}
\right),
\qquad
\lambda
=
4.
$$

Furthermore, we successively calculate
\begin{gather*}
\bm{A}\bm{B}
=
\left(
\begin{array}{rrr}
0 & 2 & 5
\\
3 & 0 & 5
\\
2 & -1 & 3
\end{array}
\right),
\qquad
\bm{A}\bm{B}^{2}
=
\left(
\begin{array}{rrr}
4 & -2 & 4
\\
4 & 1 & 4
\\
2 & 0 & 3
\end{array}
\right),
\\
\bm{A}^{2}\bm{B}
=
\left(
\begin{array}{rrr}
4 & 6 & 9
\\
6 & 4 & 8
\\
5 & 3 & 6
\end{array}
\right),
\qquad
\theta
=
4.
\end{gather*}

We now evaluate the matrices
$$
\theta^{-1}\bm{A}\oplus\bm{B}
=
\left(
\begin{array}{rrr}
0 & -2 & 1
\\
0 & -1 & 2
\\
-1 & -3 & -1
\end{array}
\right),
\quad
(\theta^{-1}\bm{A}\oplus\bm{B})^{2}
=
\left(
\begin{array}{rrr}
0 & -2 & 1
\\
1 & -1 & 1
\\
-1 & -3 & 0
\end{array}
\right),
$$
and then find
$$
(\theta^{-1}\bm{A}\oplus\bm{B})^{\ast}
=
\left(
\begin{array}{rrr}
0 & -2 & 1
\\
1 & 0 & 2
\\
-1 & -3 & 0
\end{array}
\right),
\qquad
\theta^{-1}\bm{p}\oplus\bm{g}
=
\left(
\begin{array}{c}
2
\\
2
\\
3
\end{array}
\right).
$$

Returning to the usual notation, we have the solution
\begin{align*}
x_{1}
&=
\max(u_{1},u_{2}-2,u_{3}+1),
\\
x_{2}
&=
\max(u_{1}+1,u_{2},u_{3}+2),
\\
x_{3}
&=
\max(u_{1}-1,u_{2}-3,u_{3}),
\end{align*}
where the components of the vector $\bm{u}$ are given by the inequalities
$$
u_{1}
\geq
2,
\qquad
u_{2}
\geq
2,
\qquad
u_{3}
\geq
3.
$$

\section*{Acknowledgments}

The author sincerely thanks two anonymous reviewers for their insightful comments, valuable suggestions, and corrections. He is extremely grateful for pointing out the incomplete argument of Theorem~\ref{T-xAxxpqxc} in the first draft, which has given rise to a substantial reconstruction of the theorem and related results in the revised manuscript.

\bibliographystyle{abbrvurl}
\bibliography{Extremal_properties_of_tropical_eigenvalues_and_solutions_to_tropical_optimization_problems}

\begin{thebibliography}{10}

\bibitem{Akian2007Maxplus}
M.~Akian, R.~Bapat, and S.~Gaubert.
\newblock Max-plus algebra.
\newblock In L.~Hogben, editor, {\em Handbook of Linear Algebra}, Discrete
  Mathematics and Its Applications, pages 25\--1--25\--17. Taylor and Francis,
  Boca Raton, FL, 2007.
\newblock \href {http://dx.doi.org/10.1201/9781420010572.ch25}
  {\path{doi:10.1201/9781420010572.ch25}}.

\bibitem{Butkovic2010Maxlinear}
P.~Butkovi\v{c}.
\newblock {\em Max-linear Systems: Theory and Algorithms}.
\newblock Springer Monographs in Mathematics. Springer, London, 2010.
\newblock \href {http://dx.doi.org/10.1007/978-1-84996-299-5}
  {\path{doi:10.1007/978-1-84996-299-5}}.

\bibitem{Carre1971Analgebra}
B.~A. Carr{\'e}.
\newblock An algebra for network routing problems.
\newblock {\em IMA J. Appl. Math.}, 7(3):273--294, 1971.
\newblock \href {http://dx.doi.org/10.1093/imamat/7.3.273}
  {\path{doi:10.1093/imamat/7.3.273}}.

\bibitem{Cuninghamegreen1979Minimax}
R.~Cuninghame-Green.
\newblock {\em Minimax Algebra}, volume 166 of {\em Lecture Notes in Economics
  and Mathematical Systems}.
\newblock Springer, Berlin, 1979.

\bibitem{Cuninghamegreen1962Describing}
R.~A. Cuninghame-Green.
\newblock Describing industrial processes with interference and approximating
  their steady-state behaviour.
\newblock {\em Oper. Res. Quart.}, 13(1):95--100, 1962.

\bibitem{Demeulemeester2002Project}
E.~L. Demeulemeester and W.~S. Herroelen.
\newblock {\em Project Scheduling: A Research Handbook}.
\newblock International Series in Operations Research and Management Science.
  Springer, 2002.

\bibitem{Elsner2004Maxalgebra}
L.~Elsner and P.~{van den Driessche}.
\newblock Max-algebra and pairwise comparison matrices.
\newblock {\em Linear Algebra Appl.}, 385(1):47--62, 2004.
\newblock \href {http://dx.doi.org/10.1016/S0024-3795(03)00476-2}
  {\path{doi:10.1016/S0024-3795(03)00476-2}}.

\bibitem{Elsner2010Maxalgebra}
L.~Elsner and P.~{van den Driessche}.
\newblock Max-algebra and pairwise comparison matrices, {II}.
\newblock {\em Linear Algebra Appl.}, 432(4):927--935, 2010.
\newblock \href {http://dx.doi.org/10.1016/j.laa.2009.10.005}
  {\path{doi:10.1016/j.laa.2009.10.005}}.

\bibitem{Engel1975Diagonal}
G.~M. Engel and H.~Schneider.
\newblock Diagonal similarity and equivalence for matrices over groups with 0.
\newblock {\em Czechoslovak Math. J.}, 25(3):389--403, 1975.

\bibitem{Giffler1963Scheduling}
B.~Giffler.
\newblock Scheduling general production systems using schedule algebra.
\newblock {\em Naval Res. Logist. Quart.}, 10(1):237--255, September 1963.
\newblock \href {http://dx.doi.org/10.1002/nav.3800100119}
  {\path{doi:10.1002/nav.3800100119}}.

\bibitem{Golan2003Semirings}
J.~S. Golan.
\newblock {\em Semirings and Affine Equations Over Them: Theory and
  Applications}, volume 556 of {\em Mathematics and Its Applications}.
\newblock Springer, New York, 2003.

\bibitem{Gondran2008Graphs}
M.~Gondran and M.~Minoux.
\newblock {\em Graphs, Dioids and Semirings: New Models and Algorithms},
  volume~41 of {\em Operations Research/Computer Science Interfaces}.
\newblock Springer US, 2008.
\newblock \href {http://dx.doi.org/10.1007/978-0-387-75450-5}
  {\path{doi:10.1007/978-0-387-75450-5}}.

\bibitem{Heidergott2006Maxplus}
B.~Heidergott, G.~J. Olsder, and J.~van~der Woude.
\newblock {\em Max-plus at Work: Modeling and Analysis of Synchronized
  Systems}.
\newblock Princeton Series in Applied Mathematics. Princeton University Press,
  Princeton, 2006.

\bibitem{Krivulin2012Acomplete}
N.~Krivulin.
\newblock A complete closed-form solution to a tropical extremal problem.
\newblock In S.~Yenuri, editor, {\em Advances in Computer Science}, volume~5 of
  {\em Recent Advances in Computer Engineering Series}, pages 146--151. WSEAS
  Press, 2012.
\newblock \href {http://arxiv.org/abs/1210.3658} {\path{arXiv:1210.3658}}.

\bibitem{Krivulin2012Anew}
N.~Krivulin.
\newblock A new algebraic solution to multidimensional minimax location
  problems with \uppercase{C}hebyshev distance.
\newblock {\em WSEAS Trans. Math.}, 11(7):605--614, 2012.
\newblock \href {http://arxiv.org/abs/1210.4770} {\path{arXiv:1210.4770}}.

\bibitem{Krivulin2012Solution}
N.~Krivulin.
\newblock Solution to an extremal problem in tropical mathematics.
\newblock In G.~L. Litvinov, V.~P. Maslov, A.~G. Kushner, and S.~N. Sergeev,
  editors, {\em Tropical and Idempotent Mathematics: International Workshop},
  pages 132--139, Moscow, 2012.

\bibitem{Krivulin2012Atropical}
N.~Krivulin.
\newblock A tropical extremal problem with nonlinear objective function and
  linear inequality constraints.
\newblock In S.~Yenuri, editor, {\em Advances in Computer Science}, volume~5 of
  {\em Recent Advances in Computer Engineering Series}, pages 216--221. WSEAS
  Press, 2012.
\newblock \href {http://arxiv.org/abs/1212.6106} {\path{arXiv:1212.6106}}.

\bibitem{Krivulin2013Amultidimensional}
N.~Krivulin.
\newblock A multidimensional tropical optimization problem with nonlinear
  objective function and linear constraints.
\newblock {\em Optimization}, 2013.
\newblock \href {http://arxiv.org/abs/1303.0542} {\path{arXiv:1303.0542}},
  \href {http://dx.doi.org/10.1080/02331934.2013.840624}
  {\path{doi:10.1080/02331934.2013.840624}}.

\bibitem{Krivulin2014Aconstrained}
N.~Krivulin.
\newblock A constrained tropical optimization problem: {C}omplete solution and
  application example.
\newblock In G.~L. Litvinov and S.~N. Sergeev, editors, {\em
  Tropical/Idempotent Mathematics and Applications}, volume 616 of {\em
  Contemp. Math.} American Mathematical Society, Providence, RI, 2014.
\newblock \href {http://arxiv.org/abs/1305.1454} {\path{arXiv:1305.1454}}.

\bibitem{Krivulin2004Onsolution}
N.~K. Krivulin.
\newblock On solution of linear vector equations in idempotent algebra.
\newblock In M.~K. Chirkov, editor, {\em Mathematical Models. Theory and
  Applications. Issue~5}, pages 105--113. St.~Petersburg University,
  St.~Petersburg, 2004.
\newblock (in Russian).

\bibitem{Krivulin2005Evaluation}
N.~K. Krivulin.
\newblock Evaluation of bounds on the mean rate of growth of the state vector
  of a linear dynamical stochastic system in idempotent algebra.
\newblock {\em Vestnik St. Petersburg Univ. Math.}, 38(2):42--51, 2005.

\bibitem{Krivulin2006Solution}
N.~K. Krivulin.
\newblock Solution of generalized linear vector equations in idempotent
  algebra.
\newblock {\em Vestnik St. Petersburg Univ. Math.}, 39(1):16--26, March 2006.

\bibitem{Krivulin2009Methods}
N.~K. Krivulin.
\newblock {\em Methods of Idempotent Algebra for Problems in Modeling and
  Analysis of Complex Systems}.
\newblock Saint Petersburg University Press, St.~Petersburg, 2009.
\newblock (in Russian).

\bibitem{Krivulin2009Onsolution}
N.~K. Krivulin.
\newblock On solution of a class of linear vector equations in idempotent
  algebra.
\newblock {\em Vestnik St.~Petersburg University. Ser.~10. Applied Mathematics,
  Informatics, Control Processes}, (3):64--77, 2009.
\newblock (in Russian).

\bibitem{Krivulin2011Anextremal}
N.~K. Krivulin.
\newblock An extremal property of the eigenvalue for irreducible matrices in
  idempotent algebra and an algebraic solution to a \uppercase{R}awls location
  problem.
\newblock {\em Vestnik St. Petersburg Univ. Math.}, 44(4):272--281, 2011.
\newblock \href {http://dx.doi.org/10.3103/S1063454111040078}
  {\path{doi:10.3103/S1063454111040078}}.

\bibitem{Litvinov2007Themaslov}
G.~Litvinov.
\newblock Maslov dequantization, idempotent and tropical mathematics: A brief
  introduction.
\newblock {\em J. Math. Sci. (N. Y.)}, 140(3):426--444, 2007.
\newblock \href {http://arxiv.org/abs/math/0507014}
  {\path{arXiv:math/0507014}}, \href
  {http://dx.doi.org/10.1007/s10958-007-0450-5}
  {\path{doi:10.1007/s10958-007-0450-5}}.

\bibitem{Pandit1961Anew}
S.~N.~N. Pandit.
\newblock A new matrix calculus.
\newblock {\em J. SIAM}, 9(4):632--639, 1961.
\newblock \href {http://dx.doi.org/10.1137/0109052}
  {\path{doi:10.1137/0109052}}.

\bibitem{Romanovskii1964Asymptotic}
I.~V. Romanovski{\u\i}.
\newblock Asymptotic behavior of dynamic programming processes with a
  continuous set of states.
\newblock {\em Soviet Math. Dokl.}, 5(6):1684--1687, 1964.

\bibitem{Tkindt2006Multicriteria}
V.~T{\textquoteright}kindt and J.-C. Billaut.
\newblock {\em Multicriteria Scheduling: Theory, Models and Algorithms}.
\newblock Springer, Berlin, 2006.

\bibitem{Vorobjev1963Theextremal}
N.~N. Vorob{'}ev.
\newblock The extremal matrix algebra.
\newblock {\em Soviet Math. Dokl.}, 4(5):1220--1223, 1963.

\end{thebibliography}

\end{document}